\documentclass[a4paper, 12pt]{amsart}
\usepackage{mathtools, amssymb, amsthm, graphicx, amsmath, xcolor}
\usepackage[hidelinks]{hyperref}
\usepackage[centering, heightrounded]{geometry}
\usepackage[normalem]{ulem}

\usepackage[
    backend=biber,
    style=ams-numeric,
    giveninits=true,
    maxnames=9,
]{biblatex}
\renewbibmacro*{bbx:dashcheck}[2]{#2}
\AtEveryBibitem{%
  \clearfield{issn}%
  \clearfield{isbn}%
  \clearfield{url}%
  \clearfield{doi}%
  \clearfield{number}%
}
\DeclareFieldFormat[article]{volume}{\mkbibbold{#1}}
\renewcommand{\a}{\alpha}
\renewcommand{\b}{\beta}
\renewcommand{\l}{\lambda}

\newcommand{\lr}[1]{\langle #1 \rangle}
\newcommand{\di}{\displaystyle}
\newcommand{\C}{\mathbb{C}}
\newcommand{\R}{\mathbb{R}}

\newcommand{\N}{\mathbb{N}}

\title[Defocusing NLS with point interaction]{Standing waves for defocusing nonlinear Schr\"odinger equations with point interaction}
\author[N. Fukaya]{Noriyoshi Fukaya}
\address{Regional ICT Research Center for Human, Industry and Future, 
The University of Shiga Prefecture,
Shiga 522-8533, Japan
\newline\indent
Osaka Central Advanced Mathematical Institute, 
Osaka Metropolitan University, 
Osaka 558-8585, Japan}
\email{fukaya.n@e.usp.ac.jp}

\author[Y. Osada]{Yuki Osada}
\address{Department of Mathematics, Faculty of Science Division I, Tokyo University of Science,
Tokyo 162-8601, Japan
}
\email{
yukiosada59@gmail.com}

\author[M. Rastrelli]{Mario Rastrelli}
\address{Dipartimento di Matematica, Universita Degli Studi di Pisa, Largo `
Bruno Pontecorvo, 5, 56127, Pisa, Italy and Department of Applied Physics, Waseda University Tokyo 169-8555, Japan}
\email{mario.rastrelli@phd.unipi.it}

\newtheorem{theorem}{Theorem}[section]
\newtheorem{lemma}[theorem]{Lemma}
\newtheorem{proposition}[theorem]{Proposition}
\newtheorem{corollary}[theorem]{Corollary}
\theoremstyle{remark}
\newtheorem{remark}[theorem]{Remark}

\begin{document}
\allowdisplaybreaks

\begin{abstract}
     We consider standing waves of the nonlinear Schrödinger equation
$i\partial_t u = -\Delta_\alpha u + |u|^{p-1}u$
in the defocusing case in dimensions $N=2$ and $N=3$. Here, $-\Delta_\alpha$ denotes the Laplacian with a point interaction. This operator is bounded from below by a negative constant; consequently, unlike in the free case, the associated energy functional admits non-trivial minimizers. We establish existence and uniqueness of standing waves, and prove further qualitative properties, including radial symmetry, positivity, and  stability. Moreover, we build an appropriate functional space for the zero-mass  case and establish sharp decay estimates in this case. 
\end{abstract}

\maketitle

\section{Introduction}
We consider the following defocusing nonlinear Schr\"{o}dinger equation:
\begin{equation} \label{NLS}
    i\partial_t u
    =-\Delta_\alpha u
    +|u|^{p-1}u,\quad 
    (t, x)\in\R\times \R^N,
\end{equation} 
where $u$ is the complex valued unknown function of $(t, x)$, $N\in\{2, 3\}$ is the spatial dimension, $-\Delta_\alpha$ is the Schr\"odinger operator with a point interaction at the origin with strength parameter
\begin{equation}\label{eq:cond-al}
    \alpha
    \in \begin{cases}
    \hfil \R, & N=2,
\\  (-\infty, 0), & N=3,
    \end{cases}
\end{equation}
and
\begin{equation}\label{eq:cond-p}
    1<p<
    \begin{cases}
    \infty, & N=2,
\\  \hfil 2, & N=3.
    \end{cases}
\end{equation}


With the notation $-\Delta_\alpha$, we refer to the one-parameter family of self-adjoint operators that extends 
\[  -\Delta|_{C^\infty_{\mathrm{c}}(\mathbb{R}^N\backslash\{0\})}
    \]
in dimensions $N=2,3$. 
The description of the domain and the action were obtained in \cite{AGKH05}, using the the decomposition of $L^2(\mathbb{R}^N)$ into radial and spherical parts and the Weyl's limit-point and limit-circle theory \cite{ReedSimon}. 
Fixing $\lambda>0$, we call with $G_\lambda$ the fundamental $L^2$-solution of the standard Laplacian:
\[  (\lambda-\Delta)G_\lambda
    =\delta_0\ \mathrm{in}\ \mathcal{D}',\]
where $\delta_0$ is the Dirac delta function with its support at the origin. The function $G_\lambda $ is radial, stays  in $ H^{s}(\mathbb{R}^N)$ for every $s<2-N/2$, has an exponential decay at infinity. Moreover, $G_\lambda$ is singular at the origin, with logarithmic singularity for $N=2$ and $|x|^{-1}$-singularity for $N=3$. 

The domain of $-\Delta_\alpha$ and the action are characterized in the following way:

\begin{gather}\label{eq.domain}
    D(-\Delta_\alpha)=\Bigl\{u\in L^2(\mathbb{R}^N)\,\Big|\;u=f_\lambda+cG_\lambda,\: f_\lambda\in H^2(\mathbb{R}^N),
    (\alpha+\beta(\lambda))c=f_\lambda(0)\Bigr\},
\\  (\lambda-\Delta_\alpha)u=(\lambda-\Delta)f_\lambda,
\end{gather}
with $\lambda>0$ and 
\begin{align*}
\b(\l) = 
\left\{\begin{alignedat}{2}
   &\frac{1}{2 \pi} \biggl(\gamma+\log \Bigl(\frac{\sqrt{\l}}{2}\Bigr) 
\biggr), & \quad
   &N=2,
\\ &\frac{\sqrt{\l}}{4 \pi}, & \quad
   &N=3,
\end{alignedat}\right.
\end{align*}
where $\gamma$ is the Euler-Mascheroni constant. 
Thanks to the asymptotic expansion in zero of $G_\lambda$ it can be easily proved that the decomposition in regular part and singular part and the action are independent of the choice of the parameter $\lambda>0$. 

Under \eqref{eq:cond-al}, there exists a unique $e_\alpha<0$ such that $\alpha+\beta(-e_\alpha)=0$. More precisely,
\begin{equation}
    e_\alpha=
    \left\{\begin{alignedat}{2}
    &\mathopen{}-4e^{-4\pi\alpha-2\gamma}, & \quad 
    &N=2,
\\  &\mathopen{}-(4\pi\alpha)^2, & \quad 
    &N=3,\: \alpha<0.
    \end{alignedat}\right.
\end{equation}
For convenience, we introduce the notation
\begin{equation}
    \omega_\alpha 
    :=|e_\alpha|.
\end{equation}
We note that, if $\lambda\ne \omega_\alpha$, then any element $u\in D(-\Delta_\alpha)$ can be decomposed as
\[  u=f+\frac{f(0)}{\alpha+\beta(\lambda)}G_\lambda \]
for some $f\in H^2(\R^2)$.

It is known that $e_\alpha$ is the unique negative eigenvalue of $-\Delta_\alpha$, that is,
\[
\sigma_{\mathrm{p}}(-\Delta_\alpha)=
\begin{cases}
\emptyset, & N=3,\ \alpha\geq0,\\
\{e_\alpha\}, & \text{otherwise}.
\end{cases}
\]
Moreover,
\[
\sigma_{\mathrm{ess}}(-\Delta_\alpha)
=\sigma_{\mathrm{ac}}(-\Delta_\alpha)
=[0,\infty).
\]
Therefore,
\begin{equation}
     \sigma(-\Delta_\alpha)=
    \left\{\begin{aligned}
    &\mathopen{}[0,\infty),& 
    &N=3,\ \alpha\geq0,
\\  &\mathopen{}\{e_\alpha\}\cup[0,\infty),& 
    &\text{otherwise}.
    \end{aligned}\right.
\end{equation}
Accordingly, in what follows we assume that $\alpha$ lies in the range where a simple negative eigenvalue exists, as in \eqref{eq:cond-al}.
The corresponding normalized eigenfunction for the eigenvalue $e_\alpha$ is 
\begin{equation} \label{eq:chial}
    \chi_\alpha
    =\frac{G_{-e_\alpha}}{\|G_{-e_\alpha}\|_{L^2}}.
\end{equation}

The energy space $H^1_\alpha(\R^N):=D[-\Delta_\alpha]$ is the domain of the closure of the quadratic form $\langle-\Delta_\alpha u, u\rangle$ and is characterized as 
\begin{equation}\label{eq.quadratic domain}
    H^1_\alpha(\mathbb{R}^N)
    =\{f+cG_\lambda\,|\;
    f\in H^1(\mathbb{R}^N),\: c\in\mathbb{C}\},
\end{equation}
where, with $\langle\cdot,\cdot\rangle$, we denote 
\begin{equation}
    \langle u,v\rangle=\int_{\mathbb{R}^N}u(x)\overline{v(x)} dx.
\end{equation}
Moreover, the form is bounded from below $\langle-\Delta_\alpha u,u\rangle\geq e_\alpha\|u\|_{L^2}$ so, for every $\lambda>\omega_\alpha$, we can extend it on $H^1_\alpha(\R^N)$ as 
\begin{equation}
    \langle(-\Delta_\alpha+\lambda) u,u\rangle= \|\nabla f\|_{L^2}^2+\lambda\|f\|_{L^2}^2+(\alpha+\beta(\lambda))|c|^2
\end{equation}
for $u=f+cG_\lambda\in H_\alpha^1(\R^N)$. This positive quadratic form defines $\lambda$-dependent norms on the space $H^1_\alpha(\R^N)$ as
\begin{equation}
    \|u\|_{H^1_\alpha,\lambda}:=\sqrt{\langle(-\Delta_\alpha +\lambda)u,u\rangle}.
\end{equation}
For every $\lambda,\mu>\omega_\alpha$, $\|u\|_{H^1_\alpha,\lambda}$ and $\|u\|_{H^1_\alpha,\mu}$ are equivalent, so we fix  
\begin{equation}
    \|u\|_{H^1_\alpha}:=\|u\|_{H^1_\alpha,1+\omega_\alpha}.
\end{equation}
The study of perturbed spaces was extended to the fractional cases $H^s_\alpha(\mathbb{R}^N)$ in \cite{GMS18,GMS24} and also in the $L^p$ cases with the spaces $H^{s,p}_\alpha(\mathbb{R}^N)$ in  \cite{GR25,GR25a}. 

Regarding the Cauchy problem 
\begin{equation}\label{eq:Cauchy}
    \left\{\begin{aligned}
        &i\partial_tu=-\Delta_\alpha u \pm|u|^{p-1}u, &(t,x)\in \mathbb{R}\times\mathbb{R}^N\\
        &u(0)=u_0,
    \end{aligned}\right.
\end{equation}
that covers both the focusing and defocusing cases, the first results of local well-posedness in $D(-\Delta_\alpha)$ can be found in \cite{CFN21} for both dimensions $N=2,3$. They also prove global existence with the conservation along the solutions the energy 
\begin{equation}
    \begin{aligned}
    E_\pm(u)
   :={}& \frac{1}{2} \lr{-\Delta_\a u,u} \pm \frac{1}{p+1} \|u\|_{L^{p+1}}^{p+1}
\\  ={}& \frac{1}{2} (\|\nabla f_\l\|_{L^2}^2 + \l \|f_\l\|_{L^2}^2 + (\a + \b(\l)) |c|^2 - \l \|u\|_{L^2}^2) \pm \frac{1}{p+1} \|u\|_{L^{p+1}}^{p+1},\\
    \end{aligned}
\end{equation}
and the mass $\|u\|_{L^2}$.

For $N=2$, in \cite{FGI22}, local existence and uniqueness of solutions to \eqref{eq:Cauchy} in the energy space $H^1_\alpha(\mathbb{R}^2)$ were proved by using the abstract theory in \cite{OSY12} and the Strichartz estimates obtained in \cite{CMY19}. In \cite{GR25}, continuous dependence on the initial data was also established by means of the Strichartz. We note that, for $N=3$ and $1<p<2$, the local well-posedness of \eqref{eq:Cauchy} in the energy space $H^1_\alpha(\R^3)$ can be proved by using the Strichartz estimates in \cite{DMSY18} together with the above methods.
{ To be more precise, when we talk about $H^1_\alpha$-solutions we refer to the following results.
\begin{proposition}
    Let $\alpha\in \mathbb{R}$ and $p>1$ for $N=2$ and $1<p<2$ for $N=3$. For every $u_0\in H^1_\alpha(\mathbb{R}^N)$ there exists a unique maximal solution
    \begin{equation}
        u\in C((-T_{min},T_{max}),H^1_\alpha(\mathbb{R}^N))\cap C^1((-T_{min},T_{max}),H^{-1}_\alpha(\mathbb{R}^N))
    \end{equation}
    of \eqref{eq:Cauchy}  
    with initial data $u(0)=u_0$, where $H^{-1}_\alpha(\mathbb{R}^N)$ is the dual space of $H^{1}_\alpha(\mathbb{R}^N)$.
\end{proposition}
}

In this paper we study standing waves for \eqref{NLS} of the form
\[  u(t, x)
    =e^{i\omega t}\phi_\omega(x),  \]
where $\omega\in \R$ and $\phi$ solves the equation
\begin{equation} \label{SP}
    (\omega-\Delta_\alpha)\phi _\omega
    +|\phi_\omega|^{p-1}\phi_\omega 
    =0,\quad 
    x\in\R^N. 
\end{equation}
The literature on the focusing case (corresponding to negative nonlinearity in \eqref{eq:Cauchy}) is extensive, starting from the unperturbed problem treated in \cite{BC81, CL82}, where it is proved that standing waves of the form $e^{i\omega t}\phi_\omega$ for the equation
\begin{equation}
    i\partial_t u = -\Delta u - |u|^{p-1}u
\end{equation}
with ground-state profile $\phi_\omega$ are orbitally stable for every $\omega>0$ if $1 < p < 1 + 4/N$, and are unstable for every $\omega >0$ if $p \ge 1 + 4/N$.
Concerning the point-interaction Laplacian $-\Delta_\alpha$, in \cite{FGI22} existence and qualitative properties of ground states are established in dimension $N=2$ for $\omega > -e_\alpha$. A stability and instability picture similar to the classical case is recovered: for frequencies $\omega$ close to $-e_\alpha$, the standing waves are stable for every $p > 1$, whereas for large $\omega$, they are stable if $1<p\leq3$ and unstable if $p>3$.  Note that, in the critical case $p=3$, stability holds due to the attractive nature of the point interaction. In \cite{FN23}, blow-up is proved for $\omega > -e_\alpha$ when $p > 3$. The existence and properties of ground states under a fixed-mass constraint are established in dimension $N=2$ in \cite{ABCT22} and in dimension $N=3$ in \cite{ABCT}. We also mention \cite{Fukaya25}, where uniqueness and nondegeneracy are obtained via the Pohozaev identity for $N=2$.

{
The connection between the perturbed and the unperturbed elliptic problems associated with \eqref{eq:Cauchy} has been recently investigated in \cite{BNS26}. In particular, the authors prove the equivalence, in dimensions $N=2,3$, for $p>1$ and $\omega>0$, between the problem
\begin{equation}
    (\omega-\Delta_\alpha)u \pm |u|^{p-1}u = 0, 
    \quad x \in \R^N,
\end{equation}
and the corresponding unperturbed equation posed on the punctured space
\begin{equation}
    (\omega-\Delta)u \pm |u|^{p-1}u = 0, 
    \quad x \in \R^N \setminus \{0\}.
\end{equation}
Related problems have also been considered in different settings. In \cite{OP25}, a system of weakly coupled elliptic equations involving $-\Delta_\alpha$ is analyzed, while in \cite{BGR26} a parabolic problem with point interaction is studied.

These results highlight the strong interplay between the perturbed and unperturbed frameworks, while also showing that the presence of the point interaction leads to genuinely new features at the variational level. In this direction, we show that if a negative eigenvalue $e_\alpha$ exists, then for each $\omega \in [0, |e_\alpha|)$ there exists a unique standing wave, and all such standing waves are orbitally stable. 
}
The existence of a negative eigenvalue plays a crucial role in the variational structure of the problem and in the bifurcation of standing waves. Our analysis is inspired by the paper of Kaminaga and Ohta \cite{KO09}. They studied the one-dimensional defocusing nonlinear Schrödinger equation with an attractive delta potential, namely, the equation with the operator $-\partial_x^2-\delta_0$, which corresponds to the one-dimensional case of our equation \eqref{NLS}. In particular, they proved that standing waves are stable for all $\omega$, including the case $\omega=0$. The zero-mass case $\omega=0$ has also been extensively studied in the context of nonlinearities of mixed defocusing–focusing type, derivative nonlinear Schrödinger equations, and various systems; see, for example, \cite{FH21, FH25, FHI17, FHI24, G18, H22, LN25, NOW17, KW18}.

We now state our main results. The action functional associated with \eqref{SP} is given by
\begin{align*}
    S_\omega(u)
    &:=E(u)
    +\frac{\omega}{2}\|u\|_{L^2}^2,
    \quad u\in H_\alpha^1(\R^N),
\end{align*}
{where $E(u) := \frac{1}{2}\langle -\Delta_\alpha u, u\rangle 
    +\frac{1}{p+1}\|u\|_{L^{p+1}}^{p+1}.$}
If $\phi\in H_\alpha^1(\R^N)$, then $\phi$ solves \eqref{SP} if and only if $S_\omega'(\phi)=0$. This formulation, based on $H^1$, works well for applying variational methods in the case $\omega>0$. On the other hand, in the zero-mass case $\omega=0$, since $S_0=E$ does not contain the $L^2$ norm, this approach does not work in this setting. Therefore, we modify it so that it is based on $L^{p+1}\cap \dot{H}^1$. This choice is natural in view of both the quadratic and the nonlinear parts of the energy functional. Now we define the function spaces 
\begin{equation}\label{eq:defX0}
    X_{0} 
    :=\{f+qG_\lambda\,|\; f\in (L^{p+1}\cap \dot{H}^1)(\R^N),\: q\in\mathbb{C}\}, 
\end{equation}
where 
\begin{align*}
(L^{p+1}\cap \dot{H}^1)(\R^N) = \{f \in L^{p+1}(\R^N) \,|\; \nabla f \in L^2(\R^N)\}
\end{align*}
and equipped the norm
\begin{align*}
\|f\|_{L^{p+1}\cap \dot{H}^1} = \|f\|_{L^{p+1}} + \|\nabla f\|_{L^2}. 
\end{align*}
We note that $X_0$ does not depend on the choice of $\lambda>0$. Since the space $X_0$ is larger than the function space $H_\alpha^1(\R^N)$, we need to extend the domain of the quadratic form $u\mapsto \langle -\Delta_\alpha u, u\rangle$. Our first result concerns the extension of this quadratic form.

\begin{proposition}\label{p.X0}
The operator $-\Delta_\alpha$ admits an extension $-\widetilde{\Delta}_\alpha$ to $X_0$ in the sense of bilinear forms as follows:
\begin{equation}
    \begin{aligned}
    -\widetilde{\Delta}_\alpha\colon X_0 &\to X_0^*,\\
    u&\mapsto \langle -\widetilde{\Delta}_\alpha u, \cdot\rangle_{X_0^*,X_0},
    \end{aligned}
\end{equation}
where $X_0^*$ denotes the dual space of $X_0$, and
\begin{equation} \label{eq:a(u, v)qua}
    \langle -\widetilde{\Delta}_\alpha u, v\rangle_{X_0^*,X_0}
    :=\langle\nabla f_\lambda,\nabla g_\lambda\rangle
    + (\alpha + \beta(\lambda)+\lambda\|G_\lambda\|_{L^2}^2)c \overline{d}
    -\lambda\overline{d}\langle u, G_\lambda\rangle
    -\lambda c\overline{\langle v, G_\lambda\rangle},
\end{equation}
for every $u=f_\lambda+cG_\lambda$ and $v=g_\lambda+dG_\lambda$ in $X_0$.
In particular, for any $u, v\in H_\alpha^1(\R^N)$, the identity
\[
\langle -\widetilde{\Delta}_\alpha u, v\rangle_{X_0^*,X_0}
=
\langle -\Delta_\alpha u, v\rangle
\]
holds.
\end{proposition}

In what follows, for notational simplicity, we also write the extended operator $-\widetilde{\Delta}_\alpha$ as $-\Delta_\alpha$ and the dual coupling $\langle\cdot,\cdot\rangle_{X_0^*,X_0}$ as $\langle\cdot,\cdot\rangle$. We introduce the action functional associated with \eqref{SP} for $\omega=0$ by 
\[  S_0(u)
    :={\frac{1}{2}}\langle -\Delta_\alpha u, u\rangle 
    +\frac{1}{p+1}\|u\|_{L^{p+1}}^{p+1},\quad 
    u\in X_0.   \]
We denote the set of all nontrivial solutions of \eqref{SP} by 
\[  \mathcal{A}_\omega 
    :=\begin{cases}
    \{\phi \in H_\alpha^1(\R^N)\,|\;
    \phi \ne 0,\: S_\omega'(\phi) = 0\}, &
    \omega>0
\\  \{\phi \in X_0\,|\;
    \phi \ne 0,\: S_0'(\phi) = 0\}, &
    \omega=0. 
    \end{cases} \]
Now we state our main results on the properties of standing waves.

\begin{theorem} \label{theo1}
For $0\le \omega<\omega_\alpha$, $\mathcal{A}_\omega\ne\emptyset$. 
\end{theorem}

\begin{theorem} \label{thm:uniq}
For $0\le \omega<\omega_\alpha$, there exists the unique positive, radial, and decreasing function $\phi_\omega$ such that 
\[  \mathcal{A}_\omega 
    =\{e^{i\theta}\phi_\omega\,|\; 
    \theta\in \R\}. \]
\end{theorem}

\begin{theorem} \label{nthm:decay}
{Let $\phi_0$ be the positive radially symmetric solution of \eqref{SP} with $\omega = 0$ whose uniqueness and existence are guaranteed by Theorem~\ref{thm:uniq}.}  Then there exist $c, C>0$ such that 
\[  cr^{-\frac{2}{p-1}}
    \le\phi_0(r)
    \le Cr^{-\frac{2}{p-1}} \]
for all $r\ge 1$. In particular, when $N=2$, 
\[  \phi_0\in L^2(\R^2)
    \iff 1<p<3, \]
and when $N=3$, the solution $\phi_0$ belongs to $L^2(\R^3)$ for all $1<p<2$. 
\end{theorem}



\begin{theorem} \label{thm:stab1}
For $0<\omega<\omega_\alpha$, the standing wave $e^{i\omega t}\phi_\omega(x)$ is stable in $H_\alpha^1(\R^N)$. That is, for each $\varepsilon>0$ there exists $\delta>0$ such that any $H_\alpha^1$-solution $u(t)$ of \eqref{NLS} with $\|u(0)-\phi_\omega\|_{H_\alpha^1}<\delta$ satisfies
\[  \sup_{t\in\R}\inf_{\theta\in\R}\|u(t)-e^{i\theta}\phi_\omega\|_{H_\alpha^1}<\varepsilon. \]. 
\end{theorem}

\begin{theorem} \label{thm:stab2}
The standing wave $\phi_0(x)$ is stable in the following sense: for each $\varepsilon>0$ there exists $\delta>0$ such that any $H_\alpha^1$-solution $u(t)$ of \eqref{NLS}  with $\|u(0)-\phi_0\|_{X_0}<\delta$ satisfies
\[  \sup_{t\in\R}\inf_{\theta\in\R}\|u(t)-e^{i\theta}\phi_0\|_{X_0}<\varepsilon. \]
Moreover, if $N=2$ and $1<p<3$ or $N=3$, i.e., $\phi_0\in L^2(\R^N)$, then $\phi_0$ is stable in $H_\alpha^1(\R^N)$.
\end{theorem}

In our setting, a delta-type singular perturbation is {encoded} in the definition of the operator, and unlike in the one-dimensional case, it does not appear explicitly; this requires additional arguments. Furthermore, in the one-dimensional case, qualitative properties of standing waves were derived using explicit representations of solutions, whereas such representations are not available in higher dimensions. To overcome this difficulty, we adapt methods developed for defocusing equations with external potentials that are applicable in higher dimensions. {Adapting these techniques}, we establish key properties of standing waves, such as symmetry, positivity, and uniqueness, and then apply them to prove stability. 

{
The main novelty of our work lies in the structure of the associated functions. The decomposition of the quadratic form domain \eqref{eq.quadratic domain} into regular and singular components persists, leading to standing waves with a non-trivial singular part. However, these components cannot be analyzed independently; instead, one must carefully handle the competition between them throughout the analysis. A first manifestation of this is at the level of the energy functional, which features a positive contribution from the regular part and a negative one arising from the singular component.

\subsection{Organization of the paper}
The paper is organized as follows.
\begin{itemize}
    \item Section \ref{s.X0} treats the functional space $X_0$, proving Proposition \ref{p.X0}.
    \item Section \ref{s.existence} proves the existence of standing waves.
    \item Sections \ref{s.positivity}, \ref{s.uniqueness}, and \ref{s.radiality} establish positivity, uniqueness, and radiality of the standing waves.
    \item Section \ref{s.asymptotics} determines their {decay estimate.}
    \item Finally, Section \ref{s.stability} is devoted to stability.
\end{itemize}

}

\section{Laplacian with point interaction on the space $X_0$}\label{s.X0}
This Section is devoted to an appropriate description of how to extend the operator $\Delta_\alpha$ to the space $X_0$ defined in \eqref{eq:defX0}. Firstly, considering 
\begin{equation}\label{eq.1}
    u=f_\lambda+c G_\lambda= f_\mu+ c G_\mu,
\end{equation}
with 
\begin{equation}\label{eq.2}
    f_\mu=f_\lambda+c(G_\lambda-G_\mu),
\end{equation}
we see that $X_0$ does not depend on the choice of $\lambda$, because $G_\lambda-G_\mu\in H^1(\mathbb{R}^N)$. Moreover, we note    $f_\lambda\in L^{p+1}(\mathbb{R}^N)$, because $G_\lambda\in L^{p+1}(\mathbb{R}^N)$.

We can equip the space $X_0$ with the norms 
\begin{equation}
    \|u\|_{X_0,\lambda}=\|\nabla f_\lambda\|_{L^2}+|c|+\|u\|_{L^{p+1}}.
\end{equation}
From relation \eqref{eq.2}, we get 
\begin{equation}
    \|\nabla f_\mu\|_{L^2}\leq \|\nabla f_\lambda\|_{L^2}+|c|\|\nabla (G_\lambda-G_\mu)\|_{L^2},
\end{equation}
that provides the equivalence of the norms 
\begin{equation}
    \|u\|_{X_0,\lambda}\sim\|u\|_{X_0,\mu},
\end{equation}
so, from now on, we will omit the dependence of $\lambda$.
From the Sobolev embedding $H^1(\mathbb{R}^N)\hookrightarrow L^{p+1}(\mathbb{R}^N)$, we get the continuous dense embedding $H^1_\alpha(\mathbb{R}^N)\hookrightarrow X_0$: 
\begin{equation}
    \|u\|_{X_0}=\|\nabla f_\lambda\|_{L^2}+|c|+\|u\|_{L^{p+1}}\lesssim \|f_\lambda\|_{H^1}+|c|\sim\|u\|_{H^1_\alpha}.
\end{equation}
This is crucial because we will define the operator $\Delta_\alpha$ on $X_0$ extending a suitable formula from $H^1_\alpha(\mathbb{R}^N)$.
We consider the  functional defined on $X_0$ by 
\begin{align}\label{def.scalar product}
Q(u):
=  \|\nabla f_\mu\|_{L^2}^2  + (\a + \b(\mu)+\mu\|G_\mu\|_{L^2}^2) |c|^2 -2 \mu\operatorname{Re}  \overline{c}\langle u,G_\mu\rangle .
\end{align}
If $u\in H^1_\alpha$ we have 
\begin{equation}
   \begin{aligned}
       & Q(u)=\langle-\Delta_\alpha u,u\rangle &\forall u\in H^1_\alpha.
   \end{aligned}
\end{equation}
The quantity $\lr{v,G_\mu}$ is well-defined thanks to H\"older inequality, because $G_\mu\in L^{(p+1)^\prime}$, where $(p+1)^\prime=\frac{p+1}{p}<2$.
It is immediate to see that $Q$ is quadratic:
\begin{equation}
    Q(ku)=\|k\nabla f_\mu\|_{L^2}^2  + (\a + \b(\mu)+\mu\|G_\mu\|_{L^2}^2) |kc|^2 -2 \mu\operatorname{Re}  \overline{kc}\langle ku,G_\mu\rangle= |k|^2Q(u)
\end{equation}
for every $k\in \mathbb{C}$.
Before stating the next Proposition, we need the following identities
\begin{align*}
\langle G_\lambda,G_\mu\rangle = 
\begin{cases}
\di \frac{1}{2 \pi} \log \biggl(\frac{\sqrt{\l}}{\sqrt{\mu}}
\biggr)\frac{1}{\mu-\lambda},  & N=2,\\
\di \frac{1}{4 \pi(\sqrt{\l}+\sqrt{\mu})}, & N=3.
\end{cases}
\end{align*}
 By direct computation, the following equality can be verified:
\begin{equation} \label{rem.1}
    (\lambda-\mu)\langle G_\lambda,G_\mu\rangle+\beta(\mu)
    =\beta(\lambda).
\end{equation}

\begin{proposition}
   The  quadratic functional  $Q$ is well defined on $X_0$. It does not depend on the choice of $\mu$. Moreover it is continuous on $X_0$.
\end{proposition}
\begin{proof}
We consider $u=f_\l+ cG_\l=f_\mu+c G_\mu$ as in \eqref{eq.1} and compute
    \begin{align*}
\|\nabla f_\mu\|_{L^2}^2&=\|\nabla f_\l+ c\nabla (G_\l-G_\mu)\|^2_{L^2}\\
&=\|\nabla f_\l\|_{L^2}^2+|c|^2\|\nabla (G_\l-G_\mu)\|^2_{L^2}+2\operatorname{Re} \overline{c}\langle\nabla f_\l,\nabla (G_\l-G_\mu)\rangle.
\end{align*}
Thanks to the identity $-\Delta(G_\l-G_\mu)=-\l G_\l+\mu G_\mu$ and integration by parts, we can simplify 
\begin{align*}
\|\nabla (G_\l-G_\mu)\|^2_{L^2}&=\langle\nabla (G_\l-G_\mu),\nabla (G_\l-G_\mu)\rangle\\
&=-\l\|G_\l\|_{L^2}^2+(\l+\mu)\langle G_\l, G_\mu\rangle- \mu\|G_\mu\|_{L^2}^2
\end{align*}
and 
\begin{align*}
    2\operatorname{Re} \overline{c}\langle \nabla f_\l,\nabla (G_\l-G_\mu)\rangle=&
    -2\l\operatorname{Re} \overline{c}\langle u, G_\l\rangle+ 2\l |c|^2 \|G_\l\|_{L^2}^2+ 2\mu\operatorname{Re} \overline{c}\langle u, G_\mu\rangle \\&-2\mu|c|^2\langle G_\lambda, G_\mu\rangle,
\end{align*}
so we have 
\begin{align}\label{eq. gradient norm}
    \|\nabla f_\mu\|_{L^2}^2=&\|\nabla f_\l\|_{L^2}^2+|c|^2(\l\|G_\l\|_{L^2}^2+(\l-\mu)\langle G_\l, G_\mu\rangle- \mu\|G_\mu\|_{L^2}^2)\\
    &-2\l\operatorname{Re} \overline{c}\langle u, G_\l\rangle+2\mu\operatorname{Re} \overline{c}\langle u, G_\mu\rangle.\nonumber
\end{align}
Substituting \eqref{eq. gradient norm} in \eqref{def.scalar product} and using \eqref{rem.1}
we see that
\begin{align*}
    Q(u) 
=&  \|\nabla f_\mu\|_{L^2}^2  + (\a + \b(\mu)+\mu\|G_\mu\|_{L^2}^2) |c|^2 -2 \mu\operatorname{Re}  \overline{c}\langle u,G_\mu\rangle \\
=&\|\nabla f_\l\|_{L^2}^2  + (\a + \b(\l)+\l\|G_\l\|_{L^2}^2) |c|^2 -2 \l\operatorname{Re}  \overline{c}\langle u,G_\l\rangle.
\end{align*}
The continuity is now immediate from the definitions.
\end{proof}
    
\begin{proposition}
    The quadratic functional $Q$ satisfies the parallelogram identity
    \begin{equation}
        Q(u+v)+Q(u-v)=2Q(u)+2Q(v)
    \end{equation}
    for every $u,v\in X_0$.
\end{proposition}
\begin{proof}
    In the definition of $Q$, $\|\nabla f_\lambda\|_{L^2}^2$ and $|c|^2$ are induced by an inner product, so they satisfy the parallelogram identity. For this reason, we have  only to verify that 
    $$\operatorname{Re}(\overline{c}+\overline{d})\langle u+v,G_\lambda\rangle+\operatorname{Re}(\overline{c}-\overline{d})\langle u-v,G_\lambda\rangle=2\operatorname{Re}\overline{c}\langle u,G_\lambda\rangle+2\operatorname{Re}\overline{d}\langle v,G_\lambda\rangle.$$
    We can show easily that this holds for every $u=f_\lambda+cG_\lambda$ and $v=g_\lambda+dG_\lambda$ in $X_0$.
\end{proof}

 We can now define on $X_0$ by polarization formula
 \begin{equation}
     a(u,v)=\frac{1}{4}(Q(u+v)-Q(u-v)+iQ(u+iv)-iQ(u-iv)).
 \end{equation}

\begin{proof}[Proof of Proposition \ref{p.X0}]
The operator $a:X_0\times X_0\to\mathbb{C}$ is a symmetric bilinear form such that $a(v,v)=Q(v)$ and the proof is classical.
We can compute an explicit formula for $a$, calling $k(\lambda)=\a + \b(\l)+\l\|G_\l\|_{L^2}^2$. We have 
\begin{equation}
    \begin{aligned}
        Q(u+v)=& \|\nabla f_\lambda\|^2_{L^2}+\|\nabla g_\lambda\|^2_{L^2}+2\operatorname{Re} \langle\nabla f_\lambda,\nabla g _\l\rangle + k(\lambda) (|c|^2+|d|^2+2\operatorname{Re}c\overline{d})\\
        &-2\l\operatorname{Re}\overline{c}\langle u, G_\lambda\rangle -2\l\operatorname{Re}\overline{c}\langle v, G_\l\rangle- 2\l\operatorname{Re}\overline{d}\langle u, G_\l\rangle -2\l\operatorname{Re}\overline{d}\langle v, G_\lambda\rangle,\\
    \end{aligned}
\end{equation}
and the other terms are similar. This provides 
\begin{equation} \label{eq:a(u, v)}
    a(u,v)=\langle\nabla f_\lambda,\nabla g_\lambda\rangle + k(\lambda)c \overline{d} -\lambda\overline{d}\langle u, G_\lambda\rangle-\lambda c\overline{\langle v, G_\lambda\rangle}.
\end{equation}
 It is immediate to see that 
 $$|a(u,v)|\lesssim\|u\|_{X_0}\|v\|_{X_0}$$
so, for every $u\in X_0$, the functional $a(u,\cdot):X_0\to\mathbb{C}$ is linear and continuous, and we can write $$a(u,\cdot)\in X_0^*,$$ where $X_0^*$ denotes the dual space.
In particular, if $u$ belongs to the dense space $H^1_\alpha$, we have for construction 
\begin{equation}
    \begin{aligned}
        &a(u,\cdot)=\langle-\Delta_\alpha u,\cdot\rangle  &\forall u \in H^1_\alpha,
    \end{aligned}
\end{equation}
so, by density, we can extend the definition of $-\Delta_\alpha$ on the space $X_0$ 
\begin{equation}
    \begin{aligned}
        -\Delta_\alpha:&X_0 \to X_0^*\\
        &u\mapsto a(u,\cdot). 
    \end{aligned}
\end{equation}
\end{proof}

\section{Existence of solutions for \eqref{SP}}\label{s.existence}

We introduce the notation
\[  X_{\omega}^{\mathrm{reg}}
    :=\begin{cases}
    \hfil H^1(\R^N), 
   &\omega>0
\\  (L^{p+1}\cap \dot{H}^1)(\R^N), 
   & \omega=0,
   \end{cases}\]
and 
\[  X_\omega
    :=\{f+cG_\lambda\,|\;
    f\in X_\omega^{\mathrm{reg}},\: c\in \C\}.  \]
We define the least energy level $d(\omega)$ for $\omega\ge 0$ as follows:
\begin{align*}
    d(\omega) := \inf \{S_\omega(v) \,|\; v \in X_\omega\}. 
\end{align*}

\begin{lemma} \label{lemm1}
There exist $\lambda_0, \delta, C > 0$ such that for any $\lambda>\lambda_0$ and $v=f+cG_\lambda \in X_0$, 
\begin{equation}
    S_0(v)
    \ge \delta\bigl(\|\nabla f\|_{L^2}^2+|c|^2+\|v\|_{L^{p+1}}^{p+1}\bigr)
    -C. 
\end{equation}
In particular, $d(\omega)>-\infty$. 
\end{lemma}

\begin{proof}
For $v=f+cG_\lambda\in X_0$, the explicit expression of $S_0$ is
\[  S_0(v)
    =\frac{1}{2}\|\nabla f\|_{L^2}^2
    +\frac{1}{2}\bigl(\a+\b(\lambda)+\lambda\|G_\lambda\|_{L^2}^2\bigr)|c|^2
    -\lambda\operatorname{Re}\overline{c}\langle v,G_\lambda\rangle
    +\frac{1}{p+1}\|v\|_{L^{p+1}}^{p+1}.  \]
For any $\varepsilon>0$, we estimate as follows:
\begin{align*}
    -\lambda\operatorname{Re}\overline{c}\langle v,G_\lambda\rangle
   &\ge -\lambda|c|\|v\|_{L^{p+1}}\|G_\lambda\|_{L^{(p+1)/p}}
\\ &= -\varepsilon|c|^2
    -C_{\lambda,\varepsilon}\|v\|_{L^{p+1}}^2
\\ &\ge -\varepsilon|c|^2
    -\varepsilon\|v\|_{L^{p+1}}^{p+1}
    -C_{\lambda,\varepsilon}.
\end{align*}

When $N=2$, since $\lambda\|G_\lambda\|_{L^2}^2=1/(4\pi)$, it suffices to take $\lambda > 0$ large so that $\alpha+\beta(\lambda)+1/(4\pi)>0$
and $\varepsilon>0$ small so that $1/(p+1)-\varepsilon>0$ and $\alpha+\beta(\lambda)+1/(4\pi)-2\varepsilon>0$. Then we have
\begin{align*}
    S_0(v)
   &\ge \frac{1}{2}\|\nabla f_\lambda\|_{L^2}^2
    +\frac{1}{2}\Bigl(\a+\b(\lambda)+\frac{1}{4\pi}-2\varepsilon\Bigr)|c|^2
    +\Bigl(\frac{1}{p+1}-\varepsilon\Bigr)\|v\|_{L^{p+1}}^{p+1}
    -C_{\lambda,\varepsilon}
\\ &\ge -C_{\lambda,\varepsilon}.
\end{align*}

When $N=3$, since $\lambda\|G_\lambda\|_{L^2}^2=\beta(\lambda)/2$, it suffices to take $\lambda>0$ large so that
$\alpha+3\beta(\lambda)/2>0$ and
$\varepsilon>0$ small so that $1/(p+1)-\varepsilon>0$ and $\alpha+3\beta(\lambda)/2-2\varepsilon>0$. Then we have
\begin{align*}
    S_0(v)
   &\ge \frac{1}{2}\|\nabla f_\lambda\|_{L^2}^2
    +\frac{1}{2}\Bigl(\a+\frac{3}{2}\b(\lambda)-2\varepsilon\Bigr)|c|^2
    +\Bigl(\frac{1}{p+1}-\varepsilon\Bigr)\|v\|_{L^{p+1}}^{p+1}
    -C_{\lambda,\varepsilon}
\\ &\ge -C_{\lambda,\varepsilon}.
\end{align*}

This completes the proof.
\end{proof}





\begin{lemma} \label{lemm2}
For $0 \le \omega < \omega_\alpha$, it follows that $d(\omega) < 0$. 
\end{lemma}

\begin{proof}
Let $\chi_\alpha$ be the normalized eigenfunction of $-\Delta_\alpha$ given by \eqref{eq:chial}. For $c>0$,
\begin{align*}
    S_\omega(c\chi_\alpha)
    =c^2\Bigl(-\frac{1}{2}(\omega_\alpha-\omega)
    + \frac{c^{p-1}}{p+1}\|\chi_\alpha\|_{L^{p+1}}^{p+1}\Bigr).
\end{align*}
Since $\omega_\alpha-\omega>0$, we have $S_\omega(c\chi_\alpha)<0$ for small $c>0$. This implies $d(\omega)<0$.
\end{proof}

\begin{lemma} \label{lemm3}
Let $0 \le \omega < \omega_\alpha$.  If $\{v_n\} \subset X_\omega$ is a minimizing sequence for $d(\omega)$, namely, $S_\omega(v_n) \to d(\omega)$, then up to a subsequence, there exists a minimizer $v_0 \in X_\omega$ for $d(\omega)$ such that $v_n\to v_0$ in $X_\omega$. 
\end{lemma}

\begin{proof}
From Lemma \ref{lemm1}, $S_\omega$ is bounded from below on $X_\omega$. Let $\{v_n\} \subset X_\omega$ be a minimizing sequence for $d(\omega)$, namely, $S_\omega(v_n) \to d(\omega)$, and decompose $v_n=f_\lambda+cG_\lambda$. By Lemma \ref{lemm1}, $\{v_n\}$ are bounded in $X_\omega$ for $\lambda$ sufficiently large. In particular, $f_{\lambda,n}$ are bounded in $X_\omega^{\mathrm{reg}}$ and $\{c_n\}$ is also bounded. Then, up to a subsequence, there exist 
$v_0=f_{\lambda, 0}+c_0G_\lambda\in X_\omega$ such that $v_n\rightharpoonup v_0$ weakly in $X_\omega$, i.e.,
\begin{align*}
    f_{\lambda,n} &\rightharpoonup f_{\lambda,0}\ \mathrm{weakly\ in}\ X_\omega^{\mathrm{reg}},&
    c_n &\to c_0.
\end{align*}
By the weak lower semi continuity of the norms, we have
\begin{equation}\label{eq:domlimit}
\begin{aligned}
    d({\omega}) 
   &= \liminf_{n \to \infty} S_{\omega}(v_n)
\\ &= \liminf_{n \to \infty}
    \begin{multlined}[t]
    \Bigl[\frac{1}{2} \|\nabla f_{\lambda,n}\|_{L^2}^2  +\frac{1}{2} (\a + \b(\lambda)+\lambda\|G_\lambda\|_{L^2}^2) |c_n|^2 -\lambda\operatorname{Re}  \overline{c_n}\langle v_n, G_\lambda\rangle
\\  +\frac{1}{p+1} \|v_n\|_{L^{p+1}}^{p+1}+ \frac{\omega}{2} \|v_n\|_{L^2}^2\Bigr]
    \end{multlined}
\\ &\ge\begin{multlined}[t] \frac{1}{2} \|\nabla f_{\lambda,0}\|_{L^2}^2  
    +\frac{1}{2}(\a + \b(\lambda)+\lambda\|G_\lambda\|_{L^2}^2) |c_0|^2 
    -\lambda\operatorname{Re}\overline{c_0}\langle v_0, G_\lambda\rangle
\\  +\frac{1}{p+1} \|v_0\|_{L^{p+1}}^{p+1}+ \frac{\omega}{2} \|v_0\|_{L^2}^2
    \end{multlined}
\\ &= S_\omega(v_0) \ge d(\omega).
\end{aligned}
\end{equation}
Therefore, $v_0$ is a minimizer for $d(\omega)$. Since the inequality in \eqref{eq:domlimit} is actually an equality, we have $\|\nabla f_n\|_{L^2}\to \|\nabla f_0\|_{L^2}$ and $\|v_n\|_{L^{p+1}}\to \|v_0\|_{L^{p+1}}$, and, when $\omega>0$, $\|v_n\|_{L^2}\to \|v_0\|_{L^2}$. Together with $c_n\to c_0$, we obtain $\|f_n\|_{X_\omega^{\mathrm{reg}}}\to \|f_0\|_{X_\omega^{\mathrm{reg}}}$. Since weak convergence and convergence of norms imply strong convergence in $L^q(\R^N)$ for $1<q<\infty$ (see \cite[Theorem~2.11]{LL01}), we have $f_n \to f_0$ in $X_\omega^{\mathrm{reg}}$, and hence $v_n\to v_0$ in $X_\omega$.
\end{proof}

We define the set of all minimizers for $d(\omega)$ by
\begin{equation}
    \mathcal{M}_\omega 
    :=\{v\in X_\omega\,|\;
    S_\omega(v)=d(\omega)\}. 
\end{equation}
Note that Lemma~\ref{lemm3} implies $\mathcal{M}_\omega\ne\emptyset$.

\begin{lemma} \label{lem:cne0}
If $v_0=f_0+c_0G_\lambda\in \mathcal{M}_\omega$, then $c_0\ne 0$. 
\end{lemma}

\begin{proof}
Let $v_0 = f_{\l,0} + c_0 G_\l \in X_\omega\in\mathcal{M}_\omega$. Then since $S_\omega(v_0)=d(\omega)<0$, we have
\begin{align*}
    \lambda\operatorname{Re}\overline{c_0}\langle v_0, G_\lambda\rangle
    -\frac{1}{2}(\a + \b(\lambda)+\lambda\|G_\lambda\|_{L^2}^2) |c_0|^2 
    >0.
\end{align*}
Therefore $c_0 \ne 0$.
\end{proof}

\begin{lemma} \label{lemm4}
Let $0 \le \omega < \omega_\alpha$. Then $\mathcal{M}_\omega\subset \mathcal{A}_\omega$. 
\end{lemma}

\begin{proof}
Let $v_0 \in \mathcal{M}_\omega$. Since $v_0$ is a minimizer for $d(\omega)$, it follows that $\langle S_\omega'(v_0), w\rangle=(d/dt) S_\omega(v_0 + t w)|_{t=0} = 0$ for any $w\in X_\omega$, namely, $S_\omega'(v_0) = 0$. Moreover, from Lemma \ref{lem:cne0}, we have $v_0\ne 0$. Thus, $v_0 \in \mathcal{A}_\omega$. 
\end{proof}

\begin{lemma} \label{lemm8}
Let $\phi = f + c G_\lambda \in \mathcal{A}_\omega$. Then $c \neq 0$. 
\end{lemma}

\begin{proof}
Since $S_\omega'(\phi) = 0$, then $\langle S_\omega'(\phi), \phi\rangle = 0$, that is,
\begin{align*}
\omega \|\phi\|_{L^2}^2 + \|\nabla f\|_{L^2}^2 + k(\lambda) |c|^2 - 2 \lambda \operatorname{Re} \overline{c} \lr{\phi,G_\lambda} + \|\phi\|_{L^{p+1}}^{p+1} = 0. 
\end{align*}
If we assume $c = 0$, then we have 
$\omega \|f\|_{L^2}^2 + \|\nabla f\|_{L^2}^2 + \|f\|_{L^{p+1}}^{p+1} = 0$, which implies 
$f = 0$. This contradicts $\phi \neq 0$.
\end{proof}

\begin{proof}[Proof of Theorem \ref{theo1}]
Theorem \ref{theo1} immediately follows from Lemmas \ref{lemm3} and \ref{lemm4}. 
\end{proof}

\section{Positivity of standing waves}\label{s.positivity}

In this section, we show the positivity of solutions to \eqref{SP} up to phase shift. We note that under \eqref{eq:cond-p}, $C_c^\infty(\R^N)$ is dense in $X_0^{\mathrm{reg}}\cap \dot{W}^{2, \frac{p+1}{p}}(\R^N)$ and the embedding
\begin{equation} \label{eq:embC_0}
    X_0^{\mathrm{reg}}\cap \dot{W}^{2, \frac{p+1}{p}}(\R^N)
    \hookrightarrow C_0(\R^N)
\end{equation}
holds (see Appendix~\ref{sec:embedding}), where 
\begin{align*}
&    X_0^{\mathrm{reg}}\cap \dot{W}^{2, \frac{p+1}{p}}(\R^N) 
   := \left\{f \in L^{p+1}(\R^N) \,\middle|\;
    \begin{gathered}
    \nabla f \in L^2(\R^N),
\\  \partial_j\partial_k f \in L^{\frac{p+1}{p}}(\R^N) \text{ for } j,k=1,\dots,N
    \end{gathered}\right\}
\end{align*}
equipped the norm 
\begin{align*}
&    \|f\|_{X_0^{\mathrm{reg}}\cap \dot{W}^{2, \frac{p+1}{p}}} := \|f\|_{L^{p+1}} + \|\nabla f\|_{L^2} + \sum_{j,k = 1}^{N}\|\partial_j\partial_k f\|_{L^{\frac{p+1}{p}}},
\end{align*}
and 
\[  C_0(\R^N):=\{f\in C(\R^N)\,|\; \lim_{|x|\to\infty}f(x)=0\}. \]

First, we show that any solutions for \eqref{SP} satisfy the same boundary condition as the function in the domain $D(-\Delta_\alpha)$. 

\begin{lemma}\label{lem:f(0)}
Let $v=f+cG_\lambda\in X_\omega$ satisfy $(\omega-\Delta_\alpha)v\in L^{\frac{p+1}{p}}(\R^N)$. 
Then $f\in C_0(\R^N)$ and $f(0)=(\alpha+\beta(\lambda))c$. 
\end{lemma}

\begin{proof}
We only prove the assertion in the case $\omega=0$. 
By the continuous embedding $X_0\hookrightarrow L^{p+1}(\R^N)$, we can identify $L^{\frac{p+1}{p}}(\R^N)$ with a subset of $X_0^*$ by
\[  \langle F, v\rangle_{X_0^*, X_0}
    :=\int_{\R^N} F\overline{v}\,dx
    =\langle F, v\rangle_{L^{\frac{p+1}{p}}, L^{p+1}}   \]
for $F\in L^{\frac{p+1}{p}}(\R^N)$ and $v\in X_0$. Then, by the explicit formula \eqref{eq:a(u, v)} for the bilinear form associated with $-\Delta_\alpha$, we have
\begin{align*}
    \langle -\Delta_\alpha v, \varphi\rangle_{\mathcal{D}', \mathcal{D}}
   &=\langle -\Delta_\alpha v, \varphi\rangle_{L^{\frac{p+1}{p}}, L^{p+1}}
\\ &=\langle -\Delta_\alpha v, \varphi\rangle_{X_0^*, X_0}
\\ &=\langle \nabla f, \nabla\varphi\rangle 
    -\lambda c\overline{\langle\varphi, G_\lambda\rangle}
\\ &=\langle -\Delta f-\lambda cG_\lambda, \varphi\rangle_{\mathcal{D}', \mathcal{D}},
    \quad \varphi\in \mathcal{D},
\end{align*}
where $\mathcal{D}:=C_c^\infty(\R^N)$ and $\mathcal{D}'$ denotes the space of distributions. 
This implies that $-\Delta_\alpha v=-\Delta f-\lambda cG_\lambda$ in $\mathcal{D}'$. 
Moreover, from the assumption $-\Delta_\alpha v\in L^{\frac{p+1}{p}}(\R^N)$, we obtain
\begin{align}\label{eq:DfDav}
    -\Delta f
    =-\Delta_\alpha v+\lambda cG_\lambda\in L^{\frac{p+1}{p}}(\R^N).
\end{align}
From the elliptic regularity, we have $f\in X_0^{\mathrm{reg}}\cap \dot{W}^{2, \frac{p+1}{p}}(\R^N)$, 
Therefore, \eqref{eq:embC_0} implies $f\in C_0(\R^N)$. Moreover, there exists $\{f_n\}_{n=1}^\infty\subset C_c^\infty(\R^N)$ such that $f_n\to f$ in $X_0^{\mathrm{reg}}\cap \dot{W}^{2, \frac{p+1}{p}}(\R^N)$. Therefore, by \eqref{eq:DfDav}, we obtain
\begin{align*}
    f(0)
   &=\lim_{n\to\infty} f_n(0)
    =\lim_{n\to\infty} \langle (-\Delta+\lambda)f_n, G_\lambda\rangle 
\\ &=\lim_{n\to\infty}\bigl(
    \langle -\Delta f_n, G_\lambda\rangle_{L^{\frac{p+1}{p}}, L^{p+1}}
    +\lambda\langle f_n, G_\lambda\rangle_{L^{p+1}, L^{\frac{p+1}{p}}}
    \bigr)
\\ &=\langle -\Delta f, G_\lambda\rangle_{L^{\frac{p+1}{p}}, L^{p+1}}
    +\lambda\langle f, G_\lambda\rangle_{L^{p+1}, L^{\frac{p+1}{p}}}
\\ &=\langle -\Delta_\alpha v+\lambda cG_\lambda, G_\lambda\rangle_{L^{\frac{p+1}{p}}, L^{p+1}}
    +\lambda\langle f, G_\lambda\rangle_{L^{p+1}, L^{\frac{p+1}{p}}}
\\ &=\langle -\Delta_\alpha v, G_\lambda\rangle_{X_0^*, X_0}
    +\lambda c\|G_\lambda\|_{L^2}^2
    +\lambda\langle f, G_\lambda\rangle_{L^{p+1}, L^{\frac{p+1}{p}}}
\\ &= (\alpha+\beta(\lambda))c,
\end{align*}
where we used \eqref{eq:a(u, v)} again in the last equality. 
This completes the proof.
\end{proof}

\begin{lemma}\label{lem:regofphi}
Let $0\le \omega<\omega_\alpha$ and $\phi=f+cG_\lambda\in\mathcal{A}_\omega$. 
Then $\phi\in C^2(\R^N\setminus\{0\})$, $f\in C_0(\R^N)$, and $f(0)=(\alpha+\beta(\lambda))c$. 
\end{lemma}

\begin{proof}
Since $\phi$ solves
\[  (\omega-\Delta)\phi  
    +|\phi|^{p-1}\phi
    =0  \]
on any open ball $B\subset\R^N\setminus\{0\}$, the elliptic regularity (see, e.g., 
\cite[Theorem~11.7]{LL01}) implies that $\phi\in C^2(\R^N\setminus\{0\})$. 
Moreover, since
\[
    (\omega-\Delta_\alpha)\phi=-|\phi|^{p-1}\phi\in L^{\frac{p+1}{p}}(\R^N),
\]
we obtain the assertion from Lemma~\ref{lem:f(0)}. 
\end{proof}

\begin{lemma} \label{pos.sol}
Let $0\le \omega<\omega_\alpha$ and $\phi\in\mathcal{A}_\omega$. 
Then there exists $\theta\in\R$ such that $e^{i\theta}\phi$ is positive.
\end{lemma}

\begin{proof}
Let $\lambda>\omega_\alpha$, that is, $\alpha+\beta(\lambda)>0$. 
By Lemma~\ref{lem:regofphi}, we can decompose 
\[
    \phi=f+\frac{f(0)}{\alpha+\beta(\lambda)}G_\lambda.
\]
Since $f(0)\ne 0$ by Lemma~\ref{lemm8}, we can take $\theta\in\R$ such that $e^{i\theta}f(0)>0$. 
By the gauge invariance of \eqref{SP}, we may assume without loss of generality that $f(0)>0$ and show that $\phi$ is positive.

We put $\psi:=\operatorname{Im} \phi$. 
Then $\psi$ satisfies
\begin{equation}
    (\omega-\Delta_\alpha+|\phi|^{p-1})\psi=0.
\end{equation}
In particular, since $\psi\in C^2(\R^N\setminus\{0\})$ by Lemma~\ref{lem:regofphi}, it follows that $\psi$ satisfies
\begin{equation}
    (\omega-\Delta+|\phi|^{p-1})\psi=0,
    \quad x\in\R^N\setminus\{0\}
\end{equation}
in the classical sense. 
Since $f(0)\in\R$, we have $\psi(0)=0$ and $\psi\in C(\R^N)$.

Suppose that $\psi\not\equiv 0$. 
Since $\psi(0)=0$ and $\psi(x)\to 0$ as $|x|\to\infty$, the function $|\psi|$ attains its maximum at some point $x_0\ne 0$ with $|\psi(x_0)|>0$. 
It suffices to consider the case $\psi(x_0)<0$. 
Then $\nabla\psi(x_0)=0$ and $\Delta\psi(x_0)\ge 0$. 
However, by the equation, we have
\[
    \Delta\psi(x_0)
    =(\omega+|\phi(x_0)|^{p-1})\psi(x_0)
    <0,
\]
which is a contradiction. 
Therefore, $\psi\equiv 0$, and in particular, $\phi$ is real-valued.

Next, we show that $\phi$ is nonnegative. 
If not, since $\phi(x)\to\infty$ as $|x|\downarrow 0$ and $\phi(x)\to 0$ as $|x|\to\infty$, the function $\phi$ attains its minimum at some point $x_0\ne 0$ with $\phi(x_0)<0$. 
Then the same argument as above yields a contradiction. 
Hence, $\phi\ge 0$.

Finally, we show that $\phi$ is positive. 
Let $x_0\in\R^N$ be arbitrary. 
Since $\phi(x)\to\infty$ as $|x|\downarrow 0$, we can take $x_1\in\R^N\setminus\{0\}$ and an open ball $B\subset\R^N\setminus\{0\}$ such that $x_0,x_1\in B$ and $\phi(x_1)>0$. 
Since $\phi\in C^2(B)$ and
\[
    (\omega-\Delta)\phi+|\phi|^{p-1}\phi=0,
    \quad x\in B,
\]
the strong maximum principle (see, e.g., \cite[Theorem~9.10]{LL01} implies that $\phi(x_0)>0$. 
This completes the proof.
\end{proof}

\section{Uniqueness of standing waves}\label{s.uniqueness}

\begin{lemma} \label{lem:Qnonnega}
If $v\in X_0$ satisfies $\int_{\mathbb{R}^N}v\chi_\alpha\,dx=0$, then $\langle -\Delta_\alpha v, v\rangle\ge0$.
\end{lemma}

\begin{proof}
First, for $v\in D(-\Delta_\alpha)$ satisfying $\int_{\mathbb{R}^N}v\chi_\alpha\,dx=0$, the inequality $\langle -\Delta_\alpha v, v\rangle\ge0$ follows immediately from the spectral properties.

Next, let $v\in X_0$ satisfy $\int_{\mathbb{R}^N}v\chi_\alpha\,dx=0$.
Take a sequence $\{w_n\}\subset D(-\Delta_\alpha)$ such that $w_n\to v$ in $X_0$, and define
\[  v_n
    :=w_n-\langle w_n,\chi_\alpha \rangle\chi_\alpha. \]
Then $v_n\in D(-\Delta_\alpha)$,  
$\langle v_n,\chi_\alpha\rangle=0$, and $v_n\to v$ in $X_0$.
Since $\langle -\Delta_\alpha v_n, v_n\rangle\ge0$ for all $n$, we obtain
\[  \langle -\Delta_\alpha v, v\rangle
    =\lim_{n\to\infty}\langle -\Delta_\alpha v_n, v_n\rangle
    \ge0. \]
This completes the proof.
\end{proof}

\begin{proposition} \label{lemm5}
Let $0 \le \omega < \omega_\alpha$ and $\phi_1,\phi_2 \in \mathcal{A}_\omega$ are positive solution of \eqref{SP}. Then $\phi_1 = \phi_2$ holds.
\end{proposition}

\begin{proof}[Proof of Proposition~\ref{lemm5}]
We modify the argument of \cite[Lemma~2.7]{N17}. Let $\phi_1,\phi_2\in\mathcal{A}_\omega$ be two positive solutions of~\eqref{SP}. We take $c\in\mathbb{R}$ such that
\[  \int_{\mathbb{R}^N}(c\phi_1-\phi_2)\chi_\alpha\,dx
    =0,
    \quad\text{that is,}\quad
    c:=\frac{\int_{\mathbb{R}^N}\phi_2\chi_\alpha\,dx}{\int_{\mathbb{R}^N}\phi_1\chi_\alpha\,dx}
    >0, \]
where $\chi_\alpha$ is the eigen function of $-\Delta_\alpha$ belonging to the eigenvalue $e_\alpha$. By Lemma~\ref{lem:Qnonnega}, we have
\begin{equation}\label{eq:quadposi}
    \langle -\Delta_\alpha(c\phi_1-\phi_2), c\phi_1-\phi_2\rangle\ge0.
\end{equation}
Using~\eqref{SP} and the Cauchy--Schwarz inequality
\[  2c\|\phi_1\|_{L^{p+1}}^{(p+1)/2}\|\phi_2\|_{L^{p+1}}^{(p+1)/2}
    \le c^2\|\phi_1\|_{L^{p+1}}^{p+1}
    +\|\phi_2\|_{L^{p+1}}^{p+1},    \]
we compute
\begin{align*}
    0&\le
   \langle (\omega-\Delta_\alpha)(c\phi_1-\phi_2),\, c\phi_1-\phi_2\rangle
\\ &=-2c\langle(\omega-\Delta_\alpha)\phi_1,\phi_2\rangle-c^2\|\phi_1\|_{L^{p+1}}^{p+1}
    -\|\phi_2\|_{L^{p+1}}^{p+1}
\\ &\le 2c\Bigl(
      -\langle(\omega-\Delta_\alpha)\phi_1,\phi_2\rangle
      -\|\phi_1\|_{L^{p+1}}^{(p+1)/2}
       \|\phi_2\|_{L^{p+1}}^{(p+1)/2}
    \Bigr).
\end{align*}
Since $c>0$, it follows that
\[
    \|\phi_1\|_{L^{p+1}}^{p+1}
    \|\phi_2\|_{L^{p+1}}^{p+1}
    \le \langle(\omega-\Delta_\alpha)\phi_1,\phi_2\rangle^2.
\]
We note that
\begin{equation} \label{eq:intphij}
    \int_{\mathbb{R}^N}\phi_1^p\phi_2\,dx
    =-\langle(\omega-\Delta_\alpha)\phi_1,\phi_2\rangle
    =-\langle(\omega-\Delta_\alpha)\phi_2,\phi_1\rangle
    =\int_{\mathbb{R}^N}\phi_1\phi_2^p\,dx.
\end{equation}
Therefore,
\[  \|\phi_1\|_{L^{p+1}}^{p+1}
    \|\phi_2\|_{L^{p+1}}^{p+1}
    \le
    \int_{\mathbb{R}^N}\phi_1^p\phi_2\,dx\cdot
    \int_{\mathbb{R}^N}\phi_1\phi_2^p\,dx.
\]
Combining this with Hölder's inequality
\begin{align*}
    &\int_{\mathbb{R}^N}\phi_1^p\phi_2\,dx
    \le \|\phi_1\|_{L^{p+1}}^{p}\|\phi_2\|_{L^{p+1}},&
    &\int_{\mathbb{R}^N}\phi_1\phi_2^p\,dx
    \le \|\phi_1\|_{L^{p+1}}\|\phi_2\|_{L^{p+1}}^{p},
\end{align*}
we obtain
\begin{align*}
    &\int_{\mathbb{R}^N}\phi_1^p\phi_2\,dx
    = \|\phi_1\|_{L^{p+1}}^{p}\|\phi_2\|_{L^{p+1}},&
    &\int_{\mathbb{R}^N}\phi_1\phi_2^p\,dx
    = \|\phi_1\|_{L^{p+1}}\|\phi_2\|_{L^{p+1}}^{p}.
\end{align*}
By the equality condition for Hölder's inequality
\cite[Theorem~2.3 (ii.a)]{LL01},
we conclude that $\phi_2=\lambda\phi_1$ for some $\lambda>0$. By the equality~\eqref{eq:intphij} and $p>1$, we must have $\lambda=1$.
This completes the proof.
\end{proof}
\section{Radiality {of standing waves}}\label{s.radiality}

Let $P_0$ be the radial projection defined by averaging over rotations:
\begin{align*}
(P_0 f)(x) := e_0 \int_{\mathbb{S}^{N-1}} f(|x| \omega)\, d \omega.
\end{align*}
Here $e_0$ is defined by 
\begin{align*}
e_0 = 
\begin{cases}
\frac{1}{2 \pi} & (N=2),  \\
\frac{1}{4 \pi} & (N=3), 
\end{cases}
\end{align*}
It is known that $P_0$ is the orthogonal projection from $L^2(\mathbb{R}^N)$ onto $L_{\mathrm{rad}}^2(\mathbb{R}^N)$, is a bounded operator from $(L^{p+1}\cap \dot{H}^1)(\mathbb{R}^N)$ to $(L^{p+1}\cap \dot{H}^1)_{\mathrm{rad}}(\mathbb{R}^N)$, and commutes with $\Delta$ on $H^2(\mathbb{R}^N)$.

\begin{lemma} \label{lemm6}
Let $f,g \in (L^{p+1}\cap \dot{H}^1)(\mathbb{R}^N)$. Then
\begin{align}\label{eq:Pzerofg}
\int_{\mathbb{R}^N} \nabla P_0 f \cdot \overline{\nabla P_0 g}\,dx
= \int_{\mathbb{R}^N} \nabla P_0 f \cdot \overline{\nabla g}\,dx.
\end{align}
In particular, if $f$ is radial, then
\begin{align*}
\int_{\mathbb{R}^N} \nabla f \cdot \overline{\nabla P_0 g}\,dx
= \int_{\mathbb{R}^N} \nabla f \cdot \overline{\nabla g}\,dx.
\end{align*}
\end{lemma}

\begin{proof}
First assume $f,g\in C_c^\infty(\mathbb{R}^N)$. Since $P_0$ is self-adjoint, $P_0^2=P_0$, and $P_0$ commutes with $\Delta$, integration by parts yields \eqref{eq:Pzerofg}.

For general $f,g\in (L^{p+1}\cap \dot{H}^1)(\mathbb{R}^N)$, the result follows by an approximation argument, using that $C_c^\infty(\mathbb{R}^N)$ is dense in $(L^{p+1}\cap \dot{H}^1)(\mathbb{R}^N)$ and that $P_0$ is bounded on $(L^{p+1}\cap \dot{H}^1)(\mathbb{R}^N)$.
\end{proof}

\begin{lemma} \label{lemm7}
Let $\phi \in \mathcal{A}_\omega$. Then $\phi$ is radial.
\end{lemma}

\begin{proof}
We define the least energy level $d_{\mathrm{rad}}(\omega)$ for $0 \le \omega < \omega_\alpha$ by
\begin{align*}
d_{\mathrm{rad}}(\omega) := \inf \{ S_\omega(v) \,|\; v \in X_{\omega,\mathrm{rad}} \} \ (> -\infty),
\end{align*}
where $X_{\omega,\mathrm{rad}}$ denotes the subspace of radial functions in $X_\omega$.

By the same argument as in the proofs of Lemmas \ref{lemm3} and \ref{lemm4}, there exists $v_0 \in X_{\omega,\mathrm{rad}}$ such that $v_0$ is a minimizer for $d_{\mathrm{rad}}(\omega)$ and
\begin{align}
\label{eq3}
\lr{S_\omega'(v_0), \psi} = 0 \quad 
\text{for all } \psi \in X_{\omega,\mathrm{rad}}.
\end{align}
By Lemma \ref{lemm6} and the radiality of $v_0$, we obtain
\begin{align}
\label{eq4}
\lr{S_\omega'(v_0), \psi} 
= \lr{S_\omega'(v_0), P_0 \psi} \quad 
\text{for all } \psi \in X_\omega.
\end{align}
Since $P_0 \psi \in X_{\omega,\mathrm{rad}}$ for $\psi \in X_\omega$, it follows from \eqref{eq3} and \eqref{eq4} that
\begin{align*}
\lr{S_\omega'(v_0), \psi} = 0 \quad \text{for all } \psi \in X_\omega.
\end{align*}
Therefore, $v_0 \in \mathcal{A}_\omega$. From Lemma \ref{pos.sol}, there exists $\theta_1, \theta_2\in\R$ such that $e^{i\theta_1} v_0$ and $e^{i\theta_2}\phi$ are positive functions. From Proposition \ref{lemm5}, we have $\phi = e^{i (\theta_1 - \theta_2)} v_0$. Since $v_0$ is radial, $\phi$ is also radial.
\end{proof}

\begin{proof}[Proof of Theorem \ref{thm:uniq}]
From Theorem \ref{theo1}, Lemma \ref{pos.sol} and Proposition \ref{lemm5}, 
there unique exists $\phi_\omega \in \mathcal{A}_\omega$ such that $\phi_\omega$ is positive. From Lemma \ref{lemm7}, $\phi_\omega$ is radial. 
Note that {from Lemma \ref{pos.sol}, } for all $v \in \mathcal{A}_\omega$, there exists $\theta \in \R$ such that {$e^{-i \theta} v$ is positive. Moreover, from Proposition \ref{lemm5}, we have }$e^{-i \theta} v = \phi_\omega$. 
Therefore $\mathcal{A}_\omega = \{e^{i \theta} \phi_\omega \,|\; \theta \in \R\}$. 

Now we prove $\phi_\omega$ is decreasing. Since $\phi_\omega$ is positive and radial, then $\phi_\omega$ satisfies the following ODE:
\begin{align}
\label{eq7}
        \phi_\omega''(r)+\frac{N-1}{r}\phi_\omega'(r) - \omega \phi_\omega(r) - \phi_\omega(r)^p=0, \quad r\in[R,\infty)
\end{align}
for all $R > 0$. First, we prove $\phi_\omega'(r) \le 0$  for all $r \in (R,\infty)$. Assume by contradiction, that there exists $r_0 \in (R,\infty)$ such that $\phi_\omega'(r_0) > 0$.
Since $\phi_\omega' \in C((0,\infty))$, $\phi_\omega'(r) > 0$ for $|r - r_0|$ sufficiently small. Moreover, $\phi_\omega(r) > 0$ for all $r \in (R,\infty))$ and $\phi_\omega(r) \to 0$ as $r \to \infty$. Therefore, there exists local maximum point $r_1 \in (R,\infty)$ of $\phi_\omega$ on $(R,\infty)$. 
Then we have
\begin{align*}
        0 = \phi_\omega''(r_1)+\frac{N-1}{r_1}\phi_\omega'(r_1) - \omega \phi_\omega(r_1) - \phi_\omega(r_1)^p < 0.
\end{align*}
This is a contradiction.
Finally, we prove $\phi_\omega'(r) < 0$ for all $r \in (R,\infty)$. If not, there exists $r_0 \in (R,\infty)$ such that $\phi_\omega'(r_0) = 0$. From \eqref{eq7}, we have $\phi_\omega''(r_0) > 0$. Since $\phi_\omega'' \in C((R,\infty))$, we have $\phi_\omega''(r) > 0$ for $|r - r_0|$ sufficiently small. Therefore $\phi_\omega'(r) > 0$ for $|r - r_0|$ sufficiently small and $r > r_0$. This is a contradiction. Therefore we have $\phi_\omega'(r) < 0$ for all $r \in (R,\infty)$. Since $R > 0$ is arbitrary, we obtain $\phi_\omega'(r) < 0$ for all $r \in (0,\infty)$, namely, $\phi_\omega$ is decreasing.
\end{proof}

\section{Decay of standing waves}\label{s.asymptotics}
V\'{e}ron, in \cite{Veron81} considers the  one dimensional equation
\begin{equation}\label{VeronA}
    \left\{\begin{aligned}
        &\frac{d^2u_A}{dr^2}(r)+\frac{N-1}{r}\frac{du_A}{dr}(r)-u_A|u_A|^{p-1}(r)=0, &r\in[R,\infty),\\
        &u_A(R)=A.
    \end{aligned}\right.
\end{equation}
The following results are proved.
\begin{proposition}[{\cite[Proposition 1.3]{Veron81}}] \label{prop1} Let $A$, $R>0$, and $p\geq1$  three real numbers. Then there exists a unique function $u_A$ in $L^\infty(R,\infty)$ such that  solves \eqref{VeronA} and $du_A/dr$ is in $ H^1(R,\infty;r^{N-1}dr)$.
\end{proposition} 

\begin{proposition}[{\cite[Lemma 2.2]{Veron81}}] \label{prop2} 
If $1<p<\frac{N}{N-2},$ $N\geq2$, with $\frac{N}{N-2}=\infty$ when $N=2$, the function  defined on $(0,\infty)$ as $l_{p,N}r^{-\frac{2}{p-1}}$, with
$$l_{p,N}=\left(\frac{2}{p-1}\left(\frac{2p}{p-1}-N\right)\right)^{\frac{1}{p-1}},$$
is a solution of 
\begin{equation}
\label{eq6a}
    \begin{aligned}
        &\frac{d^2u}{dr^2}+\frac{N-1}{r}\frac{du}{dr}-u^p=0, &\text{ on } (0,\infty).
    \end{aligned}
\end{equation}
  For any $A>0$, all the corresponding solutions $u_A$ of \eqref{VeronA}, given by Proposition \ref{prop1}, are equivalent to $r^{-\frac{2}{p-1}}$,
that is,
\begin{align*}
\lim_{r \to \infty} u_A(r) \, r^{\frac{2}{p-1}} \in (0,\infty).
\end{align*}
\end{proposition}

While V\'{e}ron's analysis establishes the exact asymptotic behavior for the free Laplacian, our current purposes do not need the exact limit, but only sharp upper and lower estimates. We give a simple and short proof of the integrability condition for $\phi_0$.

In what follows, we use the notation $C_0$ for the space of continuous functions that vanish at infinity. To prove Theorem~\ref{nthm:decay}, we establish the following lemma.

\begin{lemma}\label{nprop1} 
Let $u, v\in (C_0\cap C^2)([1,\infty))$ be positive solutions to 
\begin{equation}\label{Veron}
    u''+\frac{N-1}{r}u'-u^{p}=0,\quad 
    r>1,
\end{equation}
and 
\begin{equation}\label{gVeron}
    v''+\frac{N-1}{r}v'-\varepsilon^{1-p}v^{p}=0,\quad r>1
\end{equation}
with $0<\varepsilon\le1$, respectively, and $u(1)\ge v(1)$. Then $u\ge v$ on $[1, \infty)$.
\end{lemma}

\begin{proof}
Set $w:=u-v$. By the equations, we have
\begin{equation}\label{eq:keyuniq}
\begin{aligned}
    w''
    +\frac{N-1}{r}w'
   &=u^p-\varepsilon^{1-p}v^p
\\ &\le u^p-v^p,\quad 
    r>1.
\end{aligned}
\end{equation}
Since $w(1)\ge0$ and $\lim_{r\to\infty}w(r)=0$, we see that $w\ge 0$ on $[1, \infty)$. Indeed, if not, then $w$ attains a negative local minimum. At such a point, the left-hand side of \eqref{eq:keyuniq} is nonnegative, while the right-hand side is strictly negative. This is a contradiction.
\end{proof}

By using the explicit solution founded in Proposition~\ref{prop2}, we prove Theorem~\ref{nthm:decay}.

\begin{proof}[Proof of Theorem~\ref{nthm:decay}]
Put $u(r):=l_{p,N}r^{-\frac{2}{p-1}}$ and $v(r):=\varepsilon \phi_0(r)$, where $0<\varepsilon\le1$ is chosen so that $u(1)\ge v(1)$. Then since $u, v\in (C_0\cap C^2)[1,\infty)$ are positive solutions to \eqref{Veron} and \eqref{gVeron}, respectively, Lemma~\ref{nprop1} gives $v\le u$, i.e., $\phi_0(r)\le \varepsilon^{-1}l_{p,N}r^{-\frac{2}{p-1}}$ on $[1, \infty)$. If we exchange the roles of $u$ and $v$, that is, take $u(r):=\phi_0(r)$ and $v(r):=\varepsilon l_{p,N}r^{-\frac{2}{p-1}}$, we also obtain $\varepsilon l_{p,N}r^{-\frac{2}{p-1}}\le \phi_0(r)$ on $[1, \infty)$. This completes the proof.
\end{proof}

\section{Stability of standing waves}\label{s.stability}

\begin{proof}[Proof of Theorem~\ref{thm:stab1}]
Suppose that the assertion does not hold. Then there exist a sequence $\{u_n\}_{n=1}^\infty$ of solutions for \eqref{NLS} and $\{t_n\}_{n=1}^\infty$ such that $u_n(0)\to\phi_\omega$ in $H_\alpha^1(\R^N)$ and 
\begin{equation} \label{eq:stlb}
    \inf_{n\in\N}\inf_{\theta\in\R}\|u_n(t_n)-e^{i\theta}\phi_\omega\|_{H_\alpha^1}
    >0. 
\end{equation}
Since $S_\omega(u_n(t_n))=S_\omega(u_n(0))\to S_\omega(\phi_\omega) = d(\omega)$, we see that $\{u_n(t_n)\}_{n=1}^\infty$ is a minimizing sequence for $d(\omega)$. Therefore, by Lemmas \ref{lemm3} and \ref{lemm4}, there exists a subsequence of $\{u_n(t_n)\}_{n=1}^\infty$, which will be denoted by the same notation, and $\psi_0\in\mathcal{A}_\omega$ such that $u_n(t_n)\to \psi_0$ in $H^1_\alpha(\R^N)$. Moreover, by using Theorem~\ref{thm:uniq}, we obtain
\[  \inf_{\theta\in\R}\|u_n(t_n)-e^{i\theta}\phi_\omega\|_{H_\alpha^1}
    =\inf_{\psi\in\mathcal{A}_\omega}\|u_n(t_n)-\psi\|_{H_\alpha^1}
    \le\|u_n(t_n)-\psi_0\|_{H_\alpha^1}
    \to 0   \]
as $n\to\infty$, which contradicts \eqref{eq:stlb}.
\end{proof}

\begin{proof}[Proof of Theorem~\ref{thm:stab2}]
We can show the stability in $X_0$ as in the proof of Theorem~\ref{thm:stab1}. Thus, we only show the stability in $H_\alpha^1(\R^N)$ when $\phi_0\in L^2(\R^N)$. 

Suppose that $\phi_0(x)$ is unstable. Then there exists a sequence $\{u_n\}_{n=1}^\infty$ of solutions for \eqref{NLS} and $\{t_n\}_{n=1}^\infty$ such that $u_n(0)\to\phi_0$ in $H_\alpha^1(\R^N)$ and 
\begin{equation} \label{eq:stlba}
    \inf_{n\in \N}\inf_{\theta\in\R}\|u_n(t_n)-e^{i\theta}\phi_0\|_{H_\alpha^1}
    >0. 
\end{equation}
Since $S_0(u_n(t_n))=S_0(u_n(0))\to S_0(\phi_0) = d(0)$, from Lemmas \ref{lemm3} and \ref{lemm4}, we see that there exists a subsequence of $\{u_n(t_n)\}_{n=1}^\infty$, which will be denoted by the same notation, and $\psi_0\in\mathcal{A}_0$ such that $u_n(t_n)\to \psi_0$ in $X_0$. By using the conservation of $L^2$-norm, we have $\|u_n(t_n)\|_{L^2}=\|u_n(0)\|_{L^2}\to \|\phi_0\|_{L^2}$. This implies $u_n(t_n)\to \phi_0$ in $H_\alpha^1(\R^N)$.

Moreover, by using Theorem~\ref{thm:uniq}, we obtain
\[  \inf_{\theta\in\R}\|u_n(t_n)-e^{i\theta}\phi_0\|_{H_\alpha^1}
    =\inf_{\psi\in\mathcal{A}_0}\|u_n(t_n)-\psi\|_{H_\alpha^1}
    \le\|u_n(t_n)-\psi_0\|_{H_\alpha^1}
    \to 0   \]
as $n\to\infty$, which contradicts \eqref{eq:stlba}.
\end{proof}

\appendix 

\section{Embeddings for $(L^r\cap \dot{W}^{1,p})(\R^N)$}\label{sec:embedding}
In this section, we summarize embedding results for $(L^r\cap \dot{W}^{1,p})(\R^N)$ following \cite{Brezis}.

For $p\ge 1$, we define $p^*\ge 1$ such that
\[  p^*
    :=\begin{dcases}
    \frac{Np}{N-p}
   &\text{if $1\le p<N$},
\\  \infty
   &\text{if $p\ge N$}.
    \end{dcases}   \] 
We note that if $1\le p<N$, then $p^*$ satisfies 
\[  \frac{1}{p^*}
    =\frac{1}{p}
    -\frac{1}{N}.    \]
Note that $p<p^*$ and $p\mapsto p^*$ is strictly increasing.

\begin{theorem}[Density of $C_c^\infty$] \label{thm:dense}
Let $1\le p, r<\infty$ satisfy $p\le r\le p^*$. Then $C_c^\infty(\R^N)$ is dense in $(L^r\cap \dot{W}^{1, p})(\R^N)$. 
\end{theorem}

\begin{proof}
Let $u\in (L^r\cap \dot{W}^{1, p})(\mathbb{R}^N)$. Let $\{\rho_n\}_{n}$ be a sequence of mollifiers such that $\operatorname{supp}\rho_n =\overline{B(0,1/n)}$. Then $\rho_n*u\in C^\infty(\R^N)$ and 
\begin{align} \label{eq:mollconv}
    \rho_n*u\xrightarrow{n\to\infty} u \quad 
    \text{in }(L^r\cap \dot{W}^{1, p})(\R^N). 
\end{align}
Let $\zeta\in C_c^\infty(\R^N)$ be a smooth cut-off function such that $0\le \zeta\le 1$, $\zeta=1$ on $B(0,1)$, and $\zeta=0$ on $B(0,2)^c$, and set $\zeta_n(x):=\zeta(x/n)$. Then $\zeta_n(\rho_n*u)\in C_c^\infty(\R^N)$. By \eqref{eq:mollconv} and Lebesgue's dominated convergence theorem, we have
\begin{equation}\label{eq:zetarho1}
\begin{aligned} 
    \|\zeta_n(\rho_n*u)-u\|_{L^{r}}
   &\le \|\zeta_n(\rho_n*u-u)\|_{L^{r}}
    +\|\zeta_n u-u\|_{L^{r}}
\\ &\le \|\zeta\|_{L^\infty}\|\rho_n*u-u\|_{L^{r}}
    +\|\zeta_n u-u\|_{L^{r}}
\\ &\xrightarrow{n\to\infty} 0.
\end{aligned}
\end{equation}
Moreover,
\begin{align}\notag
    \|\zeta_n(\rho_n*u)-u\|_{\dot{W}^{1, p}}
   &=\|(\nabla \zeta_n)(\rho_n*u)+\zeta_n(\rho_n*\nabla u)-\nabla u\|_{L^{p}}
\\ &\label{eq:est1}
    \le \frac{1}{n}\|\nabla \zeta\|_{L^\infty}\|\rho_n*u\|_{L^{p}(n<|x|<2n)}
    +\|\zeta_n(\rho_n*\nabla u)-\nabla u\|_{L^{p}}
\end{align}
For $n<|x|<2n$, we see that $B(x, 1/n)\subset\{n-1/n<|y|<2n+1/n\}$. Indeed, if $|x-y|<1/n$, then $n-1/n<|x|-|x-y|<|y|<|x|+|x-y|<2n+1/n$. Since $r\ge p$, we can estimate 
\begin{align*}
    \|\rho_n*u\|_{L^p(n<|x|<2n)}
   &=\|\rho_n*(u1_{B(x, 1/n)})\|_{L^p(n<|x|<2n)}
\\ &\le \|\rho_n*(u 1_{[n-\frac{1}{n}<|y|<2n+\frac{1}{n}]})\|_{L^p}
\\ &\le \|u\|_{L^p(n-\frac{1}{n}\le |x|\le 2n+\frac{1}{n})}
    \lesssim n^{N(\frac1p-\frac1r)}\|u\|_{L^r(|x|\ge n-\frac{1}{n})}.
\end{align*}
Thus, the first term in \eqref{eq:est1} tends to zero if $N(\frac1p-\frac1r)-1\le 0$, i.e., $r\le p^*$. The second term in \eqref{eq:est1} tends to zero because of the same estimate as \eqref{eq:zetarho1}. 
This completes the proof.
\end{proof}

\begin{remark}
In order to establish the embedding $(L^r\cap \dot{W}^{1, p})(\R^N)\hookrightarrow L^{q}(\R^N)$ for $1\le p,r<\infty$ satisfying $p\le r\le p^*$, it suffices, by the density result in Theorem~\ref{thm:dense}, to prove the embedding inequality
\begin{equation} \label{eq:gemb}
    \|u\|_{L^q}\le C\|u\|_{L^r\cap\dot{W}^{1,p}},
\end{equation}
for all $u\in C_c^\infty(\R^N)$. 

Indeed, if \eqref{eq:gemb} holds for all $u\in C_c^\infty(\R^N)$, then for any $u\in (L^r\cap \dot{W}^{1,p})(\R^N)$, the density result implies that there exists a sequence $(u_n)\subset C_c^\infty(\R^N)$ such that $u_n\to u$ in $(L^r\cap \dot{W}^{1,p})(\R^N)$. Applying \eqref{eq:gemb} to $u=u_n-u_m$, we see that $(u_n)$ is a Cauchy sequence in $L^{q}(\R^N)$; hence $u_n\to u$ in $L^{q}(\R^N)$. 

Moreover, passing to the limit as $n\to\infty$ in \eqref{eq:gemb} with $u=u_n$, we conclude that \eqref{eq:gemb} holds for all $u\in (L^r\cap \dot{W}^{1,p})(\R^N)$.
\end{remark}

\begin{theorem}[The case $p<N$] \label{cor:emb1}
Let $1\le p<N$ and $p\le r\le p^*$. Then there exists $C>0$ such that 
\begin{equation} \label{eq:sob1}
    \|u\|_{L^{p^*}}
    \le C\|\nabla u\|_{L^{p}}
\end{equation}
for all $u\in (L^r\cap \dot{W}^{1, p})(\R^N)$.
In particular, the embedding $(L^r\cap \dot{W}^{1, p})(\R^N)\hookrightarrow L^{p^*}(\R^N)$ holds.
\end{theorem}

\begin{proof}
For $u\in C_c^\infty(\R^N)$, the inequality \eqref{eq:sob1} is proved in \cite[Theorem~9.9 (p.278)]{Brezis}.
\end{proof}

\begin{corollary} \label{cor:emb2}
Let $1\le p<N$, $p\le r\le p^*$, and $q\in[r, p^*]$. Then there exists $C>0$ such that 
\begin{equation} \label{eq:sob2}
    \|u\|_{L^{q}}
    \le \|u\|_{L^{r}}
    +C\|\nabla u\|_{L^{p}}
\end{equation}
for all $u\in (L^r\cap \dot{W}^{1, p})(\R^N)$. In particular, the embedding $(L^r\cap \dot{W}^{1, p})(\R^N)\hookrightarrow L^{q}(\R^N)$ holds.
\end{corollary}

\begin{proof}
Since $r\le q\le p^*$, the interpolation inequality and \eqref{eq:sob1} imply
\begin{align*}
    \|u\|_{L^q}
   &\le \|u\|_{L^r}^\theta\|u\|_{L^{p^*}}^{1-\theta}
    \le \|u\|_{L^r}+C\|u\|_{L^{p^*}}
    \le \|u\|_{L^r}+C\|\nabla u\|_{L^{p}}
\end{align*}   
for some $\theta\in[0,1]$. This means the assertion.
\end{proof}

\begin{theorem}[The case $p=N$] \label{cor:embpN}
Let $N\le r\le q$. Then there exists $C>0$ such that
\begin{equation} \label{eq:sobpN}
    \|u\|_{L^{q}}
    \le C(\|u\|_{L^{r}}+\|\nabla u\|_{L^{N}})
\end{equation}
for all $u\in (L^r\cap \dot{W}^{1, N})(\R^N)$.
In particular, the embedding $(L^r\cap \dot{W}^{1, N})(\R^N)\hookrightarrow L^{q}(\R^N)$ holds.
\end{theorem}

\begin{proof}
We give the proof only for $N\ge 2$. For $u\in C_c^\infty(\R^N)$, by \cite[Proof of Corollary 9.11 (p.281)]{Brezis} we have
\begin{equation}
    \|u\|_{L^{\frac{mN}{N-1}}}
    \le C\bigl(\|u\|_{L^{\frac{(m-1)N}{N-1}}}
    +\|\nabla u\|_{L^N}\bigr)
\end{equation}
for all $m\ge 1$. We first choose $m$ so that
\begin{align*}
    \frac{(m-1)N}{N-1}=r
   &\iff m=1+r\frac{N-1}{N}.
\end{align*}
Then
\begin{align*}
    \frac{mN}{N-1}
    =r+\frac{N}{N-1},
\end{align*}
and hence
\begin{equation}
    \|u\|_{L^{r+\frac{N}{N-1}}}
    \le C(\|u\|_{L^{r}}
    +\|\nabla u\|_{L^N}).
\end{equation}
Since trivially
\begin{equation}
    \|u\|_{L^{r}}
    \le \|u\|_{L^{r}}
    +\|\nabla u\|_{L^N},
\end{equation}
interpolation yields
\begin{equation}
    \|u\|_{L^{q}}
    \le C(\|u\|_{L^{r}}
    +\|\nabla u\|_{L^N}),
    \quad \forall q\in[r, r+\tfrac{N}{N-1}].
\end{equation}
Reiterating this argument with 
$m=2+r\frac{N-1}{N},\ 3+r\frac{N-1}{N},\dots$, 
we obtain \eqref{eq:sobpN} for all $q\ge r$.
\end{proof}

\begin{theorem}[The case $p>N$]\label{cor:emb3}
Let $N<p\le r$. Then there exists $C>0$ such that 
\begin{equation} \label{eq:sob3}
    \|u\|_{L^\infty}\le \|u\|_{L^r}+C\|\nabla u\|_{L^p}
\end{equation}
for all $u\in (L^r\cap \dot{W}^{1, p})(\R^N)$. In particular, the embedding $(L^r\cap \dot{W}^{1, p})(\R^N)\hookrightarrow C_0(\R^N)$ holds, where 
\[  C_0(\R^N)
    :=\{f\in C(\R^N)\,|\; 
    \lim_{|x|\to \infty}f(x)=0\}.   \]
\end{theorem}

\begin{proof}
For $u\in C_c^\infty(\R^N)$, by \cite[Proof of Theorem~9.12 (pp.~282, 283)]{Brezis} we have 
\[  |u(x)|
    \le \frac{1}{|Q|}\Bigl|\int_Q u\Bigr|
    +C\|\nabla u\|_{L^p},\quad 
    \forall x\in\R^N.  \]
where $Q$ is a cube in $\R^N$ with side length $1$ containing $x$ (in particular, $|Q|=1$). Moreover, by H\"older,
\[  \frac{1}{|Q|}\Bigl|\int_Q u\Bigr|
    \le |Q|^{-1/r}\|u\|_{L^r}.    \]
Thus, $u$ satisfies \eqref{eq:sob3}. By the density (Theorem~\ref{thm:dense}) and the fact that $C_0(\R^N)$ is closed subspace in $L^\infty(\R^N)$, we obtain the conclusion.
\end{proof}

\begin{proposition}
Let $p>1$ if $N=2$ and $1<p<2$ if $N=3$. Then the embedding
\[  (L^{p+1}\cap \dot{W}^{1, 2}\cap\dot{W}^{2, \frac{p+1}{p}})(\R^N)
    \hookrightarrow C_0(\R^N)  \]
holds.
\end{proposition}

\begin{proof}
First, Theorem~\ref{cor:embpN} and Corollary~\ref{cor:emb2} imply
\begin{equation}\label{eq:emb4}
    (L^{p+1}\cap \dot{W}^{1, 2})(\R^N)
    \hookrightarrow L^q(\R^N)
    \quad \text{for all } q\in \begin{dcases}
    [p+1, \infty) & \text{if }N=2,
\\  [p+1, 2^*] & \text{if }N=3.
    \end{dcases}
\end{equation}
We note that
\begin{align*}
    \Bigl(\frac{p+1}{p}\Bigr)^*
    =\frac{N(p+1)}{Np-(p+1)}
    =\begin{dcases}
    \frac{2(p+1)}{p-1}>2 &
    \text{if }N=2,
\\  \frac{3(p+1)}{2p-1}>3 &
    \text{if }N=3 \text{ and }1<p<2.
    \end{dcases}
\end{align*}
Since $1<\frac{p+1}{p}<N$ and $\frac{p+1}{p}\le 2\le (\frac{p+1}{p})^*$, Corollary~\ref{cor:emb2} implies
\begin{equation}\label{eq:emb5}
    (\dot{W}^{1, 2}\cap\dot{W}^{2, \frac{p+1}{p}})(\R^N)
    \hookrightarrow \dot{W}^{1, (\frac{p+1}{p})^*}(\R^N).
\end{equation}
Since $\max\{p+1, (\frac{p+1}{p})^*\}<2^*$, we can take
$q\in (\max\{p+1, (\frac{p+1}{p})^*\}, 2^*)$.
Then $N<(\frac{p+1}{p})^*<q$, and by Theorem~\ref{cor:emb3} we obtain
\begin{equation}\label{eq:emb6}
    (L^q\cap\dot{W}^{1, (\frac{p+1}{p})^*})(\R^N)
    \hookrightarrow C_0(\R^N).
\end{equation}
Combining \eqref{eq:emb4}, \eqref{eq:emb5}, and \eqref{eq:emb6} yields the assertion.
\end{proof}

\section*{Acknowledgements}
N.F. was supported by JSPS KAKENHI Grant Number JP26K00612.

Y.O. was supported by 
Research Fellowship Promoting International Collaboration,
The Mathematical Society of Japan.

\printbibliography
\end{document}